\documentclass[11pt]{amsart}
\usepackage[margin = 1in]{geometry}
\usepackage{amsfonts, amstext, amsmath, amssymb, amsthm}

\newtheorem{lemma}{Lemma}
\newtheorem{theorem}[lemma]{Theorem}
\newtheorem{proposition}[lemma]{Proposition}
\newtheorem{corollary}[lemma]{Corollary}

\usepackage{tikz}
\usetikzlibrary{patterns, snakes}

\usepackage[parfill]{parskip}

\newcommand{\RR}{\mathbf{R}}

\begin{document}

\title{Rectifiability and finite variation}
\author[Hendtlass]{Matthew Hendtlass}
 \address{School of Mathematics and Statistics,
    University of Canterbury,
    Christchurch 8041,
    New Zealand}
\email{matthew.hendtlass@canterbury.ac.nz}

\begin{abstract}
 We show that the length of a path in \( \mathbf{R}^2 \) can be computed if and only if its variation in every direction can.
\end{abstract}

\maketitle

\section{Introduction}

In this paper we explore the relationship between computing the variation of a function and computing its length; however, to get an exact correspondence we must work with paths. Our main result says, roughly, that we can compute the length of a path if and only if we can measure the length of every projection of the path.

To make our notion of `computation' precise, and relatively painless, we work in Bishop's-style constructive mathematics,\footnote{The name indicates its foremost practitioner, Errett Bishop, and not a dogmatic or puritanical nature (which would of course be Bishops or Bishops').} which is mathematics using intuitionistic logic and dependent choice,\footnote{In fact, we do not make use of any choice, so our proofs are very general and are valid, for example, in Aczel's Constructive ZF set theory or in the internal logic of any topos.} and also generally avoiding impredicativity; see \cite{BB, BV} for an introduction to Bishop's constructive mathematics. So hereafter when we make a statement like `the supremum exists' we mean that there is a constructive proof that we can calculate this supremum to arbitrary precision; this proof then embodies an algorithm for computing the supremum and guarantees that it meets its specifications. The extraction of (efficient) algorithms from constructive proofs is an open area of research (see \cite[Chapter 7]{SW} for an introduction).

Although our analysis grew out of an interest in functions and their graphs, it is necessary for us to work with more general objects. A \emph{path} in a metric space $(X, \rho)$ is a continuous function from $[0, 1]$ into $X$. If $f$ is a real-valued function on \( [0, 1] \), then \( \alpha: t \mapsto (t, f(t)) \) is a path in \( \RR^2 \) tracing the graph of \( f \). For a path $\alpha:[0, 1]\rightarrow X$ and a \emph{partition} $P: 0 = x_0 \leqslant x_1 \leqslant \cdots \leqslant x_n = 1$ of $[0, 1]$ we define
 $$
  l_P(\alpha) = \sum_{i = 0}^{n - 1} \rho((x_i, \alpha(x_i)), (x_{i + 1}, \alpha(x_{i + 1})));
 $$
$l_P(\alpha)$ is the length of the piecewise linear approximation to $\alpha$ given by sampling $\alpha$ at the points of $P$. The \emph{length} of a path $\alpha$ is then
 $$
  l(\alpha) = \sup_{P\in\mathcal{P}} l_P(\alpha),
 $$
where $\mathcal{P}$ is the set of all partitions of $[0, 1]$. If $l(\alpha)$ is a real number (that is, can be computed to arbitrary precision), then $\alpha$ is said to be \emph{rectifiable}. We will show that a path $\alpha$ in $\mathbf{R}^2$ is rectifiable whenever each projection of $\alpha$ onto a one dimensional subspace of $\mathbf{R}^2$ is rectifiable. Even if $l(\alpha)$ has not been shown to be a real number, we may write, for example, $l(\alpha) < M$ to mean that $l_P(\alpha) < M$ for all partitions $P$ of $[0, 1]$.

%
%

For a vector $w\in\mathbf{R}^2\setminus\{0\}$ the \emph{variation of} $\alpha:[0, 1]\rightarrow \mathbf{R}^2$ \emph{along} $w$ is 
 $$
  v_w(\alpha) = l(\pi_w\circ\alpha),
 $$
where $\pi_w$ is the standard projection of $\mathbf{R}^2$ onto $\mathbf{R}w$. If $w = (\cos(\theta), \sin(\theta))$, we write $v_\theta(\alpha)$ for $v_w(\alpha)$. We note that $v_{\theta_1} = v_{\theta_2}$ whenever $\theta_1 = \theta_2\ \mbox{mod } \pi$ . If $\alpha$ is the graph of a real-valued function on $[0, 1]$, then $v_{\pi / 2}(\alpha)$ is what is normally meant by the variation of $\alpha$. A path in \( \RR^2 \) is said to have \emph{finite variation in every direction} if $v_\theta$ exists for all $\theta \in \RR$.

We can now state our main result.

\bigskip
\begin{theorem}\label{main}
 A path $\alpha:[0, 1] \rightarrow \RR^2$ is rectifiable if and only if it has finite variation in every direction.
\end{theorem}

To begin, it is instructive to see a `well-behaved' function that cannot be shown constructively to be rectifiable or have finite variation. These examples are similar to those in \cite{BHP}---where it is shown that a rectifiable function has finite variation and that the converse cannot be shown constructively---and further details can be found therein. 

Let \( (a_n)_{n \geqslant 1} \) be a binary sequence with at most one non-zero term. Define 
\[ 
 f_n = \frac{1}{2^n} \inf\left\{\left\vert 2^nx - k \right\vert : k \in \mathbf{N} \right\}
\]
and let \( f = \sum_{k = 1}^\infty a_kf_k \); since \( \Vert f_n \Vert_\infty = 2^{-n} \) for each \( n \), \( f \) is well defined. 
Note that if \( a_n = 1 \) for some \( n \), then \( f = f_n \) has length \( \sqrt{2} \) and variation \( 1 \); on the other hand, if \( a_n = 0 \) for all \( n \), then \( f = 0 \) has length \( 1 \) and variation \( 0 \). In either case \( 1 \) is a Lipshitz constant of \( f \).

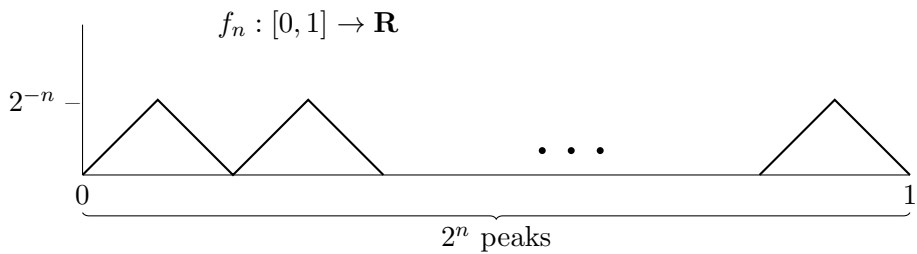
\begin{center}
\begin{tikzpicture}

\node at (4, 3) {\( {f_n: [0,1] \to \mathbf{R}} \)};

\draw (1, 1) -- (12, 1);
\draw (1, 1) -- (1, 3);

\draw[thick] (1, 1) -- (2, 2) -- (3, 1) -- (4, 2) -- (5, 1);
\draw[thick] (10, 1) -- (11, 2) -- (12, 1);

\node at (7.5, 1.3){{\Huge \( \cdots \)}};

\node[anchor = east] at (1.1, 2) {\( 2^{-n} \) --};

\node[anchor = north] at (1, 1) {\( 0 \) };
\node[anchor = north] at (12, 1) {\( 1 \) };

\draw [
    decoration ={
        brace,
        mirror
    },
    decorate
] (1, 0.5) -- (12, 0.5);

\node at (6.5, 0.15) {\( 2^n \) peaks};

\end{tikzpicture}
\end{center}

Suppose that the variation \( v_{\pi / 2}(f) \) of \( f \) exists. Either \( v_{\pi / 2}(f) > 0 \) or \( v_{\pi / 2}(f) < 1 \). In the first case, let \( P \) be a partition of \( [0, 1] \) such that \( l =  l_P(\pi_w \circ f) > 0 \) where \( w = (1, 0) \). Then there exists \( x \in P \) such that \( f(x) > l / |P| > 0 \); fix \( N \) such that \( f(x) > 2^{-n} \). If \( a_n = 0 \) for all \( n < N \), then 
 \[
   f(x) = \sum_{k = n}^\infty f_k(x) \leqslant f_{n}(x) \leqslant 2^{-n},
 \]
which is absurd; whence by examining \( a_1, a_2, \ldots, a_{n - 1} \) we can find \( m \) such that \( a_m = 1 \). On the other hand, if \( v_{\pi / 2}(f) < 1 \) we must have that \( a_n = 0 \) for all \( n \): for if \( a_n = 1 \), then \( f = f_n \) has variation \( 1 \). Altogether we have shown that if \( f \) has finite variation, then \( \forall n: a_n = 0 \) or \( \exists n: a_n = 1 \). Since in general there is no constructive algorithm to make such a decision about  \( (a_n)_{n \geqslant 1} \), there can be no algorithm to compute the variation of \( f \). In the parlance of reverse constructive mathematics \cite{HI}, `every \( f:[0, 1] \to \mathbf{R} \) has finite variation' implies the nonconstructive \emph{limited principle of omniscience}
 \begin{quote}
  LPO: For any binary sequence  \( (a_n)_{n \geqslant 1} \) either \( \forall n: a_n = 0 \) or \( \exists n: a_n = 1 \).
 \end{quote}
A similar proof shows that if  \( f \) is rectifiable, then \( \forall n: a_n = 0 \) or \( \exists n: a_n = 1 \). Finally, the function \( g:[0, 1] \rightarrow \RR \) given by \( g(x) = f(x) + x \) has variation \( 1 \), since it is non-decreasing, but is rectifiable if and only if \( \forall n: a_n = 0 \) or \( \exists n: a_n = 1 \).


The next section sketches the extension of the proof of Theorem 16 of \cite{BHP} to show that a rectifiable path has finite variation in every direction and Section \ref{fv_rect} establishes the converse.

\section{Rectifiable implies finite variation}\label{S:rect_fv}

The proof is very similar to that of \cite[Theorem 16]{BHP} and we provide only a sketch. The next result is analogous to  Lemma 15 of \cite{BHP} and makes use of natural extensions of Lemmas 13 and 14 of the same paper.

%
%


\bigskip
\begin{lemma}\label{L14}
Let \( \alpha \) be a uniformly continuous path in \( \RR^2 \) that is rectifiable. If \( P, P^\prime \) are partitions of \( [0, 1] \) with \( P \subset P^\prime \) and \( v_{\theta, P^\prime}(\alpha) - v_{\theta, P}(\alpha) > \varepsilon \) for some \( \theta \), then \( l_{P^\prime}(\alpha) - l_P(\alpha) > \sqrt{l(\alpha)^2 + \varepsilon^2} - l(\alpha) \).
\end{lemma}

\begin{proof}
Fix \( \theta \). We first consider the special case where
 \begin{itemize}
  \item[(i)] \( P \) is the trivial partition \( 0, 1 \),
  \item[(ii)] \( P^\prime \) is a strict partition \( 0 = s_0 < s_1 < \cdots < s_n = 1\) and \( \{\pi_{\theta} \circ \alpha(s_i): i \in n\} \) and \( \{\pi_{\theta + \pi / 2} \circ \alpha(s_i): i \in n\} \) are discrete.
 \end{itemize}

By (ii), the functions \( s, t: n \rightarrow \{0, 1\} \) given by
 \[
 s(i) = \begin{cases}
          -1 & \Vert \pi_\theta(\alpha(s_{i + 1}) - \alpha(s_i)) \Vert < 0\\
          \ \ 1 & \Vert \pi_\theta(\alpha(s_{i + 1}) - \alpha(s_i)) \Vert > 0
         \end{cases} 
 \]
and 
 \[
 t(i) = \begin{cases}
          -1 & \Vert \pi_{\theta + \pi / 2}(\alpha(s_{i + 1}) - \alpha(s_i)) \Vert < 0\\
          \ \ 1 & \Vert \pi_{\theta + \pi / 2}(\alpha(s_{i + 1}) - \alpha(s_i)) \Vert > 0
        \end{cases}
\]
are well defined. Let \( \beta \) be the piecewise linear path in \( \RR^2 \) defined by \( \beta(0) = 0 \) and, for \( i = 1, \ldots, n \),
\[
\beta(s_i) = \sum_{i \in i} (s(j)\pi_\theta(\alpha(s_{i + 1}) - \alpha(s_i)), t(j) \pi_{\theta + \pi / 2}(\alpha(s_{i + 1}) - \alpha(s_i))).
\]
By construction \( l_{P^\prime}(\beta) = l_{P^\prime}(\alpha) \) and \( v_{\theta, P^\prime}(\beta) = v_{\theta, P^\prime}(\alpha) = \Vert \pi_{\theta + \pi / 2}(\beta(1)) \Vert \). With 
 \[ 
  d = \Vert \pi_\theta \circ \beta(1) - \pi_\theta \circ \beta(0) \Vert 
 \] 
we have
 \begin{eqnarray*}
  l_{P^\prime}(\alpha) & = & l_{P^\prime}(\beta) \\
    & \geqslant & \sqrt{d^2 + \Vert \pi_{\theta + \pi / 2}(\beta(1)) \Vert^2} \\
    & = & \sqrt{d^2 + (v_{\theta, P^\prime}(\beta))^2} \\
    & > & \sqrt{d^2 + (v_{\theta, P}(\beta) + \varepsilon)^2}.
 \end{eqnarray*}
Thus
 \begin{eqnarray*}
  l_{P^\prime}(\alpha) - l_P(\alpha) & \geqslant & \sqrt{d^2 + (v_{\theta, P}(\beta) + \varepsilon)^2} - \sqrt{d^2 + (v_{\theta, P}(\beta))^2}\\
   & \geqslant & \sqrt{d^2 + \varepsilon^2} - d \\
   & \geqslant & \sqrt{l(\alpha)^2 + \varepsilon^2} - l(\alpha),
 \end{eqnarray*}
where the second to last inequality uses a straightforward extension of Lemma 13 of \cite{BHP} and the last inequality holds since \( d \leqslant l(\alpha) \).

The general case can then be proved in a way analogous to the proof of Lemma 15 of \cite{BHP}, making use of a suitable extension of Lemma 14 therein.
\end{proof}

\bigskip
\begin{theorem}
 If a path \( \alpha:[0, 1] \rightarrow \RR^2 \) is rectifiable, then it has finite variation in every direction.
\end{theorem}

The proof is essentially the same as that of \cite[Theorem 16]{BHP}. However, in that proof it is assumed that \( P \cup Q \) is a partition for partitions \( P, Q \); this cannot be established constructively, since partitions are given in non-decreasing order and we may not be able to order \( P \cup Q \). So we give the full proof here.

\begin{proof}
For \( \theta \in \mathbf{R} \) let
\[
S = \{v_{\theta, P}(\alpha): \text{$P$ is a partition of $[0,1]$} \}.
\]
Given real numbers $a, b$ with $a < b$, let $\varepsilon = \frac{a+b}{2}$ and let $P$ be a partition of $[0,1]$ such that
\begin{equation}
l(\alpha) - l_{P}(\alpha) < \sqrt{1+\varepsilon^{2}} - 1.  \label{F:P}
\end{equation}
Either $v_{\theta, P}(\alpha) > a$ or $v_{\theta, P}(\alpha) < (b - a)/2$. In the latter case,
consider any partition $P^{\prime}$ of $[0, 1]$. Since \( \alpha \) is continuous, we can perturb the points of \( P \) slightly, if need be, to ensure that \( x \neq x^\prime \) for all \( x \in P, x^\prime \in P^\prime \) while preserving (\ref{F:P}). Then we can form a partition with points \( P \cup P^\prime \).

By Lemma \ref{L14},
\[
v_{\theta, P\cup P^{\prime}}(\alpha) \leqslant v_{\theta, P}(\alpha) + \varepsilon < b.
\]
It follows from the constructive least-upper-bound principle (Theorem 3.2.4 of \cite{BB}) that \( v_{\theta}(\alpha) \) exists; that is, that \( \alpha \) has finite variation along \( (\cos(\theta), \sin(\theta)) \).
\end{proof}

\section{Finite variation implies rectifiable}\label{fv_rect}


For the work of this section we are in need of a little more notation. For a path $\alpha$, a partition $P: 0 = x_0 \leqslant x_1 \leqslant \cdots \leqslant x_n = 1$ of $[0, 1]$, and $i < n$ we write $l^\alpha_{P, i}$ and $\theta^\alpha_{P, i}$ for the length and angle, respectively, of the vector $(x_{i +1}, \alpha(x_{i + 1})) - (x_{i}, \alpha(x_{i}))$. So
 $$
  l_P(\alpha) = \sum_{i = 1}^{n -1} l^\alpha_{P, i}
 $$ 
and
 $$
  v_{\theta, P}(\alpha) = \sum_{i = 1}^{n -1} |\cos(\theta - \theta^\alpha_{P, i})| l^\alpha_{P, i}.
 $$

The proof of the second direction of Theorem \ref{main} has three steps:
 \begin{enumerate}
  \item show that \( \theta \mapsto v_\theta(\alpha) \) is uniformly continuous,
	\item build a partition \( P \) such that \( v_\theta(\alpha) - v_P(\alpha) \) is sufficiently small for all \( \theta \);
	\item derive a bound \( f(x) \) for \( l(\alpha) - l_P(\alpha) \) from an upper bound \( x \) of \( \{v_\theta(\alpha) - v_P(\alpha): \theta \in \RR \} \) such that \( f(x) \to 0 \) as \( x \to 0^+ \).
 \end{enumerate}
Our first lemma works toward establishing the first step.

\bigskip
\begin{proposition}\label{P:UnCts}
 For every path $\alpha$ in $\mathbf{R}^2$, every partition $P$ of $[0, 1]$, and all $\theta_1, \theta_2$,
  $$
   |v_{\theta_1, P}(\alpha) - v_{\theta_2, P}(\alpha)| \leqslant 2 l_P(\alpha) |\theta_1 - \theta_2|.
  $$
\end{proposition}

\begin{proof}
First we suppose that $\theta_1, \theta_2$ and each $\theta^\alpha_{P, i}$ is rational; so without loss of generality we may assume that $\vert \theta_1 -\theta^\alpha_{P, i} \vert \leqslant \pi / 2$ and $\vert\theta_2 -\theta^\alpha_{P, i}\vert \leqslant \pi / 2$ for each $i < n$. Then for each $i < n$ we have
 \begin{eqnarray*}
   ||\cos(\theta_1 - \theta^\alpha_{P, i})| - |\cos(\theta_2 - \theta^\alpha_{P, i})|| 
    & = & |\cos(\theta_1 - \theta^\alpha_{P, i}) - \cos(\theta_2 - \theta^\alpha_{P, i})| \\
    & = & |\cos(\theta_1)\cos(\theta^\alpha_{P, i}) + \sin(\theta_1)\sin(\theta^\alpha_{P, i}) - \\ & & \ \ \ \ \ \ \ \ \ \ \ (\cos(\theta_2)\cos(\theta^\alpha_{P, i}) + \sin(\theta_2)\sin(\theta^\alpha_{P, i}) )| \\
    & = & |(\cos(\theta_1) - \cos(\theta_2))\cos(\theta^\alpha_{P, i}) + \\ & & \ \ \ \ \ \ \ \ \ \ \ (\sin(\theta_1) - \sin(\theta_2))\sin(\theta^\alpha_{P, i})| \\
    & \leqslant & |(\cos(\theta_1) - \cos(\theta_2)| + |\sin(\theta_1) - \sin(\theta_2)| \\
    & \leqslant & 2|\theta_1 - \theta_2|.
 \end{eqnarray*}
Since $\sin, \cos$ and $|\cdot|$ are all continuous, we can drop the assumption that $\theta_1, \theta_2$ and $\theta^\alpha_{P, i}$ ($i < n$) are  rational. Thus
 \begin{eqnarray*}
  |v_{\theta_1, P}(\alpha) - v_{\theta_2, P}(\alpha)| 
    & = & \left\vert \sum_{i = 1}^{n - 1}l^\alpha_{P, i}|\cos(\theta_1 - \theta^\alpha_{P, i})| - \sum_{i = 1}^{n - 1}l^\alpha_{P, i}|\cos(\theta_2 - \theta^\alpha_{P, i})| \right\vert \\
    & = & \left\vert \sum_{i = 1}^{n - 1} l^\alpha_{P, i}(|\cos(\theta_1 - \theta^\alpha_{P, i})| - |\cos(\theta_2 - \theta^\alpha_{P, i})|) \right\vert \\
    & \leqslant & \sum_{i = 1}^{n - 1} 2 l^\alpha_{P, i} |\theta_1 - \theta_2| \\
    & = & 2 l_{P}(\alpha) |\theta_1 - \theta_2|.
 \end{eqnarray*}
\end{proof}

To finish the proof of Step 1 we must show that \( l(\alpha) \) is bounded above.

\bigskip
\begin{lemma}
For all $\gamma \in (0, \pi)$ there exists $r \in \mathbf{R}$ such that $l(\alpha) \leqslant r(v_\theta(\alpha) + v_{\theta + \gamma}(\alpha))$ for all $\theta$.
\end{lemma}

\begin{proof}
Let $r = 1 / \min\{|\cos(\theta) + \cos(\theta + \gamma)|: \theta \in [0, \pi]\}$. It suffices to show that $l_P(\alpha) \leqslant r (v_{\theta, P}(\alpha) + v_{\theta + \gamma, P}(\alpha))$ holds for all partitions $P$ of $[0, 1]$. For a partition $P: 0 = x_0 \leqslant x_1 \leqslant \cdots \leqslant x_n = 1$ we have
 \begin{eqnarray*}
  r (v_{\theta, P}(\alpha) + v_{\theta + \gamma, P}(\alpha)) & = & r\sum_{i = 1}^{n - 1}\left(|\cos(\theta - \theta^\alpha_{P, i}) l^\alpha_{P, i}| +  |\cos(\theta + \gamma - \theta^\alpha_{P, i}) l^\alpha_{P, i}|\right) \\
  & \geqslant &  r\sum_{i = 1}^{n - 1}|\cos(\theta - \theta^\alpha_{P, i}) + \cos(\theta + \gamma - \theta^\alpha_{P, i})|  l^\alpha_{P, i}\\
  & \geqslant & \sum_{i = 1}^{n - 1} l^\alpha_{P, i} \ \ = \ \ l_P(\alpha).
 \end{eqnarray*}
\end{proof}

\bigskip
\begin{corollary}
Let $\alpha$ be a path in $\mathbf{R}^2$. If $v_{\theta_1}(\alpha)$ and $v_{\theta_2}(\alpha)$ are bounded above for distinct $\theta_1, \theta_2 \in [0, \pi)$, then $l(\alpha)$ is bounded above.
\end{corollary}

The next lemma provides us with Step 2.

\bigskip
\begin{lemma}\label{L:cty}
Let $\alpha$ be a path in $\mathbf{R}^2$ such that $v_\theta(\alpha)$ exists for all $\theta$. Then $\theta\mapsto v_\theta(\alpha)$ is uniformly continuous and for all $\varepsilon > 0$ there exists a partition $P$ of $[0, 1]$ such that $v_\theta(\alpha) - v_{\theta, P}(\alpha) < \varepsilon$ for all $\theta$.
\end{lemma}

\begin{proof}
By the previous corollary there exists an upper bound $M$ on the length of $\alpha$. Then $ |v_{\theta_1}(\alpha) - v_{\theta_2}(\alpha)| \leqslant 2 M |\theta_1 - \theta_2|$ for all $\theta_1, \theta_2$---for if not, then there exists a partition $P$ such that $|v_{\theta_1, P}(\alpha) - v_{\theta_2, P}(\alpha)| \leqslant 2 l_P(\alpha) |\theta_1 - \theta_2|$,  contradicting Proposition \ref{P:UnCts}.

Let $\delta = \varepsilon / 6M$ and let $\theta_0, \ldots, \theta_n$ be a $\delta$-approximation to $[0, \pi)$ with $\theta_0 = 0$. For each $0 \leqslant i \leqslant n$, let $P_i$ be a partition of $[0, 1]$ such that $v_{\theta_i}(\alpha) - v_{\theta_i, P_i}(\alpha) < \varepsilon / 3$; since \( \alpha \) is continuous we can assume that $P = \bigcup P_i$ is a partition. Given any $\theta$ there exists $i$ such that $|\theta - \theta_i| \ \ \mathrm{mod}  \ \ \pi < \delta$, so
 \begin{eqnarray*}
  v_\theta(\alpha) - v_{\theta, P}(\alpha) & \leqslant & |v_{\theta, P}(\alpha) - v_{\theta_i, P}(\alpha)| + |v_{\theta_i, P}(\alpha) - v_{\theta_i}(\alpha)| + |v_{\theta_i}(\alpha) - v_\theta(\alpha)| \\
  & < & 2\delta M + \varepsilon / 3 + 2 \delta M \ \ = \ \ \varepsilon.
 \end{eqnarray*}
\end{proof}

Now it just remains to establish the third step and put everything together.

\bigskip
\begin{theorem}
 If $\alpha$ is a path in $\mathbf{R}^2$ such that $v_\theta(\alpha)$ exists for each $\theta$, then $l(\alpha)$ exists.
\end{theorem}

\begin{proof}
Given $\varepsilon > 0$, let $P: 0 = x_0 \leqslant x_1 \leqslant \cdots \leqslant x_n = 1$ be a partition of $[0, 1]$ such that $v_\theta(\alpha) - v_{\theta, P}(\alpha) < \varepsilon / \pi$ for all $\theta$, and consider any extension $P^\prime$ of $P$. Let $\alpha_P$ be the piecewise linear path with $\alpha_P(x_i) = \alpha(x_i)$ for each $i$ and 
 $$
  \alpha_P(t) = \frac{t - x_i}{x_{i + 1} - x_i}\alpha(x_i) + \left(1 - \frac{t - x_i}{x_{i + 1} - x_i}\right)\alpha(x_{i + 1})
 $$ 
for $t\in (x_i, x_{i + 1})$, and let $\alpha_{P^\prime}$ be defined similarly. Then $\theta\mapsto v_{\theta}(\alpha_P)$ and $\theta\mapsto v_{\theta}(\alpha_{P^\prime})$ are uniformly continuous and 
 $$
  \int_0^\pi v_{\theta}(\alpha_P)\ d\theta > \int_0^\pi v_{\theta}(\alpha_P^\prime) - \varepsilon / \pi\ d\theta.
 $$
Now, since 
 \begin{eqnarray*}
  \int_0^\pi v_{\theta}(\alpha_P)\ d\theta & = &  \int_0^\pi v_{\theta, P}(\alpha)\ d\theta \\
    & = & \int_0^\pi \sum_{i = 0}^{n - 1} l^\alpha_{P, i} |cos(\theta - \theta^\alpha_{P, i})|\ d\theta \\
    & = & \sum_{i = 0}^{n - 1} l^\alpha_{P, i} \int_0^\pi |cos(\theta)|\ d\theta,
 \end{eqnarray*}
and similarly for $P^\prime$, we have that
 $$
  l_P(\alpha) = \sum_{i = 0}^{n - 1} l^\alpha_{P, i} > \sum_{i = 0}^{n - 1} l^\alpha_{P^\prime, i} - \varepsilon = l_{P^\prime}(\alpha) - \varepsilon.
 $$
Thus $l(\alpha) - l_P(\alpha) < \varepsilon$, as $P^\prime$ is an arbitrary extension of $P$. Since $\varepsilon > 0$ is arbitrary, $l(\alpha)$ exists.
\end{proof}

\section{Conclusion}

The function \( f \) described in the introduction (which has finite variation if and only if a given binary sequence \( (a_n)_{n\geqslant 1} \) has a 1) viewed as a path, has variation in two spread-out directions (\( \theta = 0 \) and \( \theta = \pi / 2 \)) and still cannot be shown to be rectifiable. Indeed, for each \( n \) there is a path that has variation in \( n \) spread-out directions that is not rectifiable. In contrast it follows from Theorem \ref{main} and Lemma \ref{L:cty} that if \( v_\theta(\alpha) \)  exists for each \( \theta \) in a dense subset of \( [0, \pi) \), then \( \alpha \) is rectifiable. 

A path in \( \RR^n \) has \emph{finite variation in every direction} if $v_w$ exists for all $w \in \RR^n$ with \( ||w|| = 1 \). The proof of Theroem \ref{main} can be routinely extended to establish

\bigskip
\begin{theorem}
 A path $\alpha:[0, 1] \rightarrow \RR^n$ is rectifiable if and only if it has finite variation in every direction.
\end{theorem}

\end{document}